\documentclass[a4paper,12pt,english]{article}
\usepackage{amsmath,amsfonts,amssymb}

\usepackage[utf8]{inputenc}
\usepackage[english]{babel}
\usepackage{xfrac}
\usepackage{amsthm}
\usepackage{titlesec}
\usepackage{hyperref}
\usepackage{ stmaryrd }

\usepackage{tikz}

\newtheorem{theorem}{Theorem}[section]

\newtheorem{corollary}[theorem]{Corollary}
\newtheorem{lemma}[theorem]{Lemma}

\theoremstyle{remark}
\newtheorem{remark}[theorem]{Remark}

\newcommand{\p}{{{\Psi}}}
\newcommand{\E}{\mathbb{E}}
\newcommand{\h}{\mathcal{H}}
\newcommand{\prob}{\mathbb{P}}
\newcommand{\pP}{\mathcal{P}}
\newcommand{\ee}{\varepsilon}
\newcommand{\cB}{\mathcal{B}}
\newcommand{\CC}{\mathbb{C}}
\newcommand{\NN}{\mathbb{N}}

\DeclareMathOperator{\Tr}{Tr}
\DeclareMathOperator{\tr}{tr}
\DeclareMathOperator{\Ker}{{\ker}}

\title{A dynamical version of the SYK Model and the $q$-Brownian Motion}
\author{Miguel Pluma and Roland Speicher}
\date{}

\begin{document}
\maketitle
\begin{abstract}
We extend recent results on the asymptotic eigenvalue distribution of the SYK model 
to the multivariate case and relate the limit of a dynamical version of the SYK model with the $q$-Brownian motion, a 
non-commutative deformation of classical Brownian motion.
Furthermore, we extend the results for fluctuations to the multivariate setting and treat 
also higher correlation functions. The structure of our results for the sparse SYK random matrices resembles the formulas for higher order freeness for ordinary GUE random matrices.
\end{abstract}

\section{Introduction}
The SYK model was introduced by Sachdev and Ye \cite{Sachdev-1993} in 1993 as a model 
for a quantum random spin system and has attracted a lot of interest in the last few 
years since it was promoted in 2015 by Kitaev \cite{Kitaev, MalSt} to a simple model 
of quantum holography. The SYK model is a quantum mechanical model for $n$ interacting 
Majorana fermions with a random coupling for a $q_n$-interaction. In the original model, 
$q_n$ was independent of $n$ and equal to 4, but it has turned out that there are 
interesting and treatable limits for $n\to\infty$ if one also scales the number of 
Majoranas in the interaction term as $q_n\sim \lambda\sqrt n$. 

The SYK model is a kind of sparse random matrix model. It was observed, on various 
levels of physical and mathematical rigor (see, e.g., \cite{Verb,Erd-2014,Feng:2018zsx}), 
that the asymptotic eigenvalue distribution of the SYK model, depending on the parameter 
$q=e^{-2\lambda}$, is given by a $q$-deformation of the Gaussian distribution.
Such deformations have been considered before in various contexts in physics and 
mathematics. Most importantly, from our perspective, this distribution appears as the 
fixed time distribution of a non-commutative Brownian motion, realized on a $q$-deformed 
version of a Fock space, as considered in \cite{Mare-1991,BKS}. In this context there 
is a multivariate extension of the distribution from the fixed time random variable to 
the whole process. We want to investigate here, in how far there are multivariate 
extensions of the SYK model which match the distribution of the $q$-Brownian motion. 
It turns out that replacing the couplings in the SYK model by independent classical Brownian motions will do the job. This yields a dynamical SYK model which converges to the $q$-Brownian motion.

Our calculations are 
essentially adaptations of the calculations in \cite{Feng:2018zsx,Erd-2014} to the 
multivariate situation. In this context we also want to point out the appearance of the 
$q$-Gaussian distribution as a limit distribution of random matrix models in the papers 
\cite{SpCLT, Sniady, Parisi}.

It is not clear to us whether this multivariate versions have any physical relevance; 
but we want to point out that recently Berkooz and collaborators computed in 
\cite{Berk1,Berk2} the 2-point and 4-point function of the large $n$ double-scaled SYK 
model, by using also the combinatorics of such multivariate extensions. The problems 
they encounter there are related to the lack of a good analytic description of the 
distribution of the multivariate $q$-Gaussian distribution. We will say a few words on 
these problems in the final section of this paper.

We will also look on the multivariate extension for the calculation of fluctuations 
from \cite{Feng:2018tsw}, extending this also to higher correlation functions. It would be interesting to put these fluctuations into the 
setting of second and higher order freeness \cite{CMSS,MS}; however, as the random matrix models 
considered here are quite sparse they seem to be too far away from such a setting; in 
particular, the case $q=0$, which gives asymptotically the semicircular distribution has 
quite different fluctuations from the GUE, which is the ``canonical'' random matrix model 
for the semicircle. On the other hand, the form of our results in Theorem \ref{FluctuationsThm} shows still some features of ``partitioned permutations'', which are the main conceptual tools for dealing with higher order freeness in \cite{CMSS}. This points at the possibility that there might be a form of ``higher order freeness'' for sparse random matrices. This has to be investigated further.

What we are providing here is a dynamical version of the SYK model. For the relevance of embedding the SYK model into a process, let us emphasize  that having process versions of random matrix models or non-commutative distributions might be an advantage, even if one is only interested in the marginal distribution at one fixed time. Much of recent progress of Erd\"os, Yau and coworkers on properties of eigenvalue distributions relies on such a dynamical approach to random matrices, see for example \cite{EY}. Also the recent breakthrough of Driver, Hall, and Kemp for the calculation of Brown measures depends crucially on the study of the time evolution of the involved processes; for this, see the recent survey \cite{Hall}. We will be looking on our dynamical version of the SYK model from this point of view in forthcoming investigations.

\section{Preliminaries}
\subsection{Set partitions}
For any positive integers $k<l$ we define
$[k,l]:=\{k,k+1,\dots,l\}$.
The set $[1,k]$ for $k\geq1$, will be also denoted by $[k]$. Denote the set 
of \emph{partitions} of $[k]$ by $\pP(k)$.
This means that if $\pi\in\pP(k)$, then $\pi$ is a non empty set of subsets of $[k]$,
any pair $V,W\in\pi$ is disjoint as long as $V\neq W$,
and $[k]=\bigcup_{V\in\pi}V$. Elements in $\pi$ will be called \emph{blocks}. The set of 
partitions $\pP(k)$ has an order structure given as follows: for $\pi,\sigma\in\pP(k)$ 
we say $\pi\leq\sigma$ if every block of $\pi$ is contained in a block of $\sigma$.
With this order, $\pP(k)$ is a lattice, i.e., for $\pi,\sigma\in\pP(k)$, there exists a uniquely determined maximum, $\pi\vee\sigma$, and a uniquely determined minimum, $\pi\wedge\sigma$, in $\pP(k)$. It is common to denote by $1_k$ the partition in $\pP(k)$
with one block.

The size of a set $S$ will be denoted by $|S|$.\par
A basic ingredient in the construction of $q$-Gaussian variables is played by pair
partitions. The set of \emph{pair partitions} $\pP_2(k)$ on $[k]$ is defined as follows
$$\pP_2(k)=\{\pi\in\left.\pP(k)\right| \forall B\in\pi:\, |B|=2\}.$$
For a pair partition $\pi\in\pP_2(k)$, we will say that two different blocks 
$V=\{v_1<v_2\}\in\pi$ and $W=\{w_1<w_2\}\in\pi$ are \emph{crossing}, if 
$v_1<w_1<v_2<w_2$ or $w_1<v_1<w_2<v_2$.
If $\pi$ does not have a crossing we will say it is \emph{non-crossing}. 
Furthermore for $\pi\in\pP_2(k)$ we will denote by $cr(\pi)$ the number of crossing blocks
in $\pi$, i.e.,
$$cr(\pi):= \vert \{ \{V,W\} \mid V,W\in \pi, V\not= W, \text{$V$ and $W$ are crossing}\}\vert.$$

\subsection{Notation for products of non-commutative variables}\label{ProdNotation}
Consider a family $\{X_s\}_{s\in A}$ of non-commutative variables. Given 
$\alpha:[k]\rightarrow A$ we denote
\begin{equation}\label{N1}
 X_\alpha:=X_{\alpha(1)}\cdots X_{\alpha(k)}.
\end{equation}
In case we have several families of non-commutative variables $\{X^{(r)}_s\}_{s\in A}$ 
for $r\in B$ we will also use similar notation. That is, given $\alpha:[k]\rightarrow A$ 
and $\varepsilon:[k]\rightarrow B$ we denote
\begin{equation}\label{N2}
 X^\varepsilon_\alpha:=
     X^{(\varepsilon(1))}_{\alpha(1)}\cdots X^{(\varepsilon(k))}_{\alpha(k)}.
\end{equation}
It will be useful to specify the functions $\alpha:[k]\rightarrow A$ via partitions. 
For this purpose we define for every function $\alpha:[k]\rightarrow A$ between 
discrete spaces
\begin{equation*}\label{kernel}
 \Ker \alpha:=
     \{\alpha^{-1}(a)|\:a\in A\mbox{ and }\alpha^{-1}(a)\neq\emptyset\}.
\end{equation*}
Given $B\in\ker\alpha$, we will denote the common value of $\alpha$ in $B$ by
\begin{equation}\label{N3}
\alpha(B):=\alpha(b_1)=\cdots=\alpha(b_r),
\end{equation}
where $B=\{b_1,\dots,b_r\}$.

\subsection{The SYK model}
The Sachdev--Ye--Kitaev (SYK) model was introduced in \cite{Sachdev-1993} and \cite{Kitaev} 
as a model for a quantum magnet and quantum holography, respectively.
Let $n$ be an even number and consider $\psi_1,\dots,\psi_n$ Majorana fermions, i.e. 
variables which fulfill the following relations
\begin{equation}\label{fermions}
  \psi_i\psi_j+\psi_j\psi_i=2\delta_{ij}.
\end{equation}
These variables can be realized using Pauli matrices
\[
\begin{array}{ccc}
 \sigma_1=\left(\begin{array}{cc}
                 0 & 1\\
                 1 & 0
                \end{array}\right), &
 \sigma_2=\left(\begin{array}{cc}
                 0 & -i\\
                 i & 0
                \end{array}\right), &
 \sigma_3=\left(\begin{array}{cc}
                 1 & 0\\
                 0 & -1
                \end{array}\right),
\end{array}
\]
in the following fashion: for $n=2r$ each Majorana fermion is constructed as an 
$r$-fold tensor product
\begin{align*}
 \psi_1&= \sigma_1 \otimes        1 \otimes \cdots \otimes 1 & & &
          \psi_{r+1}&= \sigma_2 \otimes        1 \otimes \cdots \otimes 1\\
 \psi_2&= \sigma_3 \otimes \sigma_1 \otimes \cdots \otimes 1 & & &
          \psi_{r+2}&= \sigma_3 \otimes \sigma_2 \otimes \cdots \otimes 1\\
                          \vdots& & ,& & &\vdots\\
 \psi_r&= \sigma_3 \otimes \sigma_3 \otimes \cdots \otimes \sigma_1 & & &
          \psi_{2r}&= \sigma_3 \otimes \sigma_3 \otimes \cdots \otimes \sigma_2
\end{align*}
where the $1$ in the tensor products represent the $2\times2$ identity matrix.
In particular, for $n=2$ the above expressions reduce to  $\psi_1=\sigma_1$ and 
$\psi_2=\sigma_2$. In this way the $\psi_1,\cdots,\psi_n$ Majorana fermions are 
realized as square matrices of size $2^{\frac {n}{2}}$.\par
The SYK model is a random linear combination of products of $q_n$ many (with $1\leq q_n\leq \frac{n}{2}$) 
Majorana fermions, and is defined as
\begin{equation}\label{SYK-model}
  H_{n,q_n}:=
    \frac{i^{\lfloor\sfrac{q_n}{2}\rfloor}}{\binom{n}{q_n}^{\sfrac{1}{2}}}
       \sum_{1\leq i_1<\cdots<i_{q_n}\leq n}
           J_{i_1,\dots,i_{q_n}} 
              \psi_{i_1}\cdots\psi_{i_{q_n}},
\end{equation}
where the random coefficients $J_{i_1,\dots,i_{q_n}}$ are independent real random 
variables with moments of all orders and 
\[\E[J_{i_1,\dots,i_{q_n}}]=0,\;\;\;\; \E[J^2_{i_1,\dots,i_{q_n}}]=1.\]
Note that the factor $i^{\lfloor\sfrac{q_n}{2}\rfloor}$ (where $i$ is here $\sqrt{-1}$) is needed to make $H_{n,q_n}$ selfadjoint.

In the main theorem about expectations in Section \ref{section:3} we do not assume the variables $J_{i_1,\dots,i_{q_n}}$ to be 
identically distributed, but we do require uniformly bounded moments. For the result 
about fluctuations and higher correlations in Section \ref{section:4} we do require identical distribution and, in order to keep things simple, we also assume a Gaussian distribution. It will be important to 
distinguish the parity of $q_n$, see Theorem \ref{MainTheorem}.\par
We are interested in the distribution of products of independent copies of the 
SYK-model. For this purpose it is convenient to
have a compact notation for (\ref{SYK-model}). This motivates the following notation: for 
$1\leq q_n\leq \frac{n}{2}$ consider the set of tuples
\begin{equation}\nonumber
 I_n:=\{(i_1,\dots,i_{q_n})| 1\leq i_1 <\dots< i_{q_n}\leq n\},
\end{equation}
and for each $R=(i_1,\dots,i_{q_n})\in I_n$ denote $J_R:=J_{i_1,\dots,i_{q_n}}$ and 
consider the new variables
\begin{equation}\label{New_variables}
   \p_R:=\psi_{i_1}\cdots\psi_{i_{q_n}} i^{\lfloor\sfrac{q_n}{2}\rfloor}.
\end{equation}
Then for $1\leq q_n\leq \frac{n}{2}$ we rewrite the SYK-model as
\begin{equation*}
  H_{n,q_n}:=\frac{1}{|I_n|^{\sfrac{1}{2}}}\sum_{R\in I_n}J_R \p_R.
\end{equation*}
We collect some properties of the variables (\ref{New_variables}) in the following
lemma. See Section \ref{ProofLemmas} for the proof.
\begin{lemma}\label{Lemma0}
For every $R,Q\in I_n$ with $R\neq Q$ we have the identities
 \begin{equation}\label{psi_square}
  \p^2_R=I,
 \end{equation}
 and 
 \begin{equation}\label{psi_commutation}
 \p_Q\p_R=(-1)^{q_n+|Q\cap R|}\p_R\p_Q,
\end{equation}
where $|Q\cap R|$ stands for the number of common indices in $Q$ and $R$. 
\end{lemma}
\noindent
So, for two different multi-indices $Q$ and $R$, the variables $\p_Q$ and $\p_R$ commute 
or anti-commute depending on the parity of $q_n$ and on the size of the intersection of 
the multi-indices. The variables (\ref{New_variables}) also behave well with respect to 
the trace, see Lemma \ref{Lemma1}. \\
For square matrices $A\in\mathbb{M}_n(\CC)$ we will denote
$$\tr(A)=\frac{1}{n}\Tr(A),$$
where $\Tr$ is the usual non-normalized trace.

\subsection{$q$-Gaussian distribution and $q$-Brownian motion}

The $q$-Gaussian distribution, also known as $q$-semicircular distribution, was 
introduced in \cite{Mare-1991,BKS} in the context of non commutative probability. 
In this section we will review some basic definitions, for this purpose we will 
mainly follow \cite{Mare-1991}. In the following $q\in[-1,1]$ is fixed.
Consider a Hilbert space $\h$. On the algebraic full Fock space
\begin{equation}\nonumber
  \mathcal{F}_{alg}(\h)=\bigoplus_{n\geq 0}\h^{\otimes n},
\end{equation}
-- where $\h^0=\mathbb{C}\Omega$ with a norm one vector $\Omega$, called ``vacuum'' 
-- we define a $q$-deformed inner product as follows: 
 \begin{equation*}
   \langle h_1\otimes\cdots\otimes h_n,g_1\otimes\cdots\otimes g_m\rangle_q 
       = \delta_{nm}
           \sum_{\sigma\in S_n}
               \prod^n_{r=1}
                   \langle h_r,g_{\sigma(r)}\rangle q^{i(\sigma)},
 \end{equation*}
where 
$$i(\sigma)=\#\{(k,l)\mid 1\leq k<l\leq n; \sigma(k)>\sigma(l)\}$$
is the number of inversions of a permutation $\sigma\in S_n$.
In \cite{Mare-1991} it was shown that this inner product is positive definite, and has a kernel only for $q=1$ and $q=-1$.

 The $q$-Fock space is then defined as the completion of the algebraic full Fock space 
 with respect to this inner product
\begin{equation*}
 \mathcal{F}_q(\mathcal{H})
    =\overline{\bigoplus_{n\geq 0}\mathcal{H}^{\otimes n}}^{\langle\cdot,\cdot\rangle_q}.
\end{equation*}
In the cases $q=1$ and $q=-1$ we have to first divide out the kernel, thus leading to the symmetric and anti-symmetric Fock space, respectively.

Now for $h\in\mathcal{H}$ we define the $q$-creation operator $a^*(h)$, given by
 \begin{align*}
 a^*(h)\Omega&=h,\\
  a^*(h)h_1\otimes\cdots\otimes h_n&=h\otimes h_1\otimes\cdots\otimes h_n.
 \end{align*}
 Its adjoint (with respect to the $q$-inner product), the $q$-annihilation operator 
 $a(h)$, is given by
  \begin{align*}
     a(h)\Omega&=0,\\
     a(h)h_1\otimes\cdots\otimes h_n&=
        \sum_{r=1}^n 
           q^{r-1} 
             \langle h,h_r\rangle 
                h_1\otimes \cdots \otimes h_{r-1}\otimes h_{r+1}\otimes\cdots \otimes h_n.
 \end{align*}

Those operators satisfy the $q$-commutation relations
$$a(f)a^*(g)-q a^*(g)a(f)=\langle f,g\rangle \cdot 1\qquad (f,g\in \mathcal{H}).$$
For $q=1$, $q=0$, and $q=-1$ this reduces to the CCR-relations, the Cuntz relations, 
and the CAR-relations, respectively. With the exception of the case $q=1$, the operators 
$a^*(f)$ are bounded.

Operators of the form  
 \begin{equation*}
  s_q(h)=a(h)+a^*(h)
 \end{equation*}
for $h\in\mathcal{H}$ are called \emph{$q$-Gaussian} (or \emph{$q$-semicircular}) 
elements.

Finally, on $\mathcal{F}_q(\mathcal{H})$ we consider the vacuum expectation state
\begin{equation*}
 \tau(T)
   =\langle \Omega,T\Omega \rangle_q,
      \quad\mbox{ for }\quad 
         T\in\mathcal{B}(\mathcal{F}_q(\mathcal{H})).
\end{equation*}

The \emph{(multivariate) $q$-Gaussian distribution} is defined as the non commutative 
distribution of a collection of $q$-Gaussians with respect to the 
vacuum expectation state. As was shown in \cite{Mare-1991}, 
the joint distribution of $s_q(f_1),\dots,s_q(f_p)$ for $f_1,\dots,f_n\in\h$ is given by the following $q$-deformed version of the Wick/Isserlis formula: for any 
$\ee:\{1,\dots,k\}\rightarrow\{1,\dots,p\}$ we have 
$$\tau\left(s_q(f_{\ee(1)})\cdots s_q(f_{\ee(k)})\right)
             =\sum_{\pi\in\pP_2(k)}
                q^{cr(\pi)}\prod_{(r,s)\in\pi}
\langle f_{\ee(r)},f_{\ee(s)}\rangle.
$$
In the case of
orthonormal 
$h_1,\dots,h_p\in\mathcal{H}$ this reduces to
        \begin{equation}\label{eq:Wick-ortho}
         \tau\left(s_q(h_{\ee(1)})\cdots s_q(h_{\ee(k)})\right)
             =\sum_{\stackrel{\pi\in\pP_2(k)}{\pi\leq \ker\ee}}
                q^{cr(\pi)}.
        \end{equation}
Of course, in the case $q=0$, $0^0$ has to be understood as 1, i.e., in this case the factor $q^{cr(\pi)}$ is suppressing all crossing pairings and the sum is effectively just running over non-crossing pair-partitions.

For $p=1$, the $q$-Gaussian distribution is a probability measure on the interval 
$[-2/\sqrt{1-q}, 2/\sqrt{1-q}]$, with analytic formulas for its density, see Theorem 
1.10 in \cite{BKS}. For the special cases $q=1$, $q=0$, and $q=-1$, this reduces to
the classical Gaussian distribution, the semicircular distribution, and the symmetric 
Bernoulli distribution on $\pm 1$, respectively.

The $q$-Brownian motion is a special process version of the $q$-Gaussian distribution. Namely, if we take as our underlying Hilbert space $\mathcal{H}=L^2(\mathbb{R}_0^+)$ and as indexing vectors the family $1_{[0,t]}$ ($t\geq 0$) of characteristic functions of intervals $[0,t]$, then the process $((S_q(t))_{t\geq 0}$ with 
\begin{equation}\label{q-BM-eq}
S_q(t)=s_q(1_{[0,t]})=a(1_{[0,t]})+a^*(1_{[0,t]})
\end{equation}
is called \emph{$q$-Brownian motion}. In the case $q=1$ it is indeed classical Brownian motion (in the sense that it has the same expectation values as classical Brownian motion), and in the case $q=0$ it is free Brownian motion.

\section{Expectations for the dynamical model}\label{section:3}

In this section we present a multi-variable as well as a dynamical version of a result from  
\cite{Feng:2018zsx} and \cite{Erd-2014}.
\begin{theorem}\label{MainTheorem}
Consider $p$ independent and identically distributed copies\\ $H_{1},\dots, H_p$ 
 of the SYK model $H_{n,q_n}$, with uniformly bounded random coefficients 
 (\ref{SYK-model}). We assume the existence of the limit
 $$\frac{q^2_n}{n}
     \rightarrow
         \lambda\in[0,\infty],
             \quad\mbox{as $n\rightarrow\infty$,}$$ 
and describe this in terms of a number $q\in [-1,1]$ in the following form:
\begin{itemize}
   \item[i)] If $(q_n)_{n\geq1}$ is a sequence of even positive integers, then 
             $q=e^{-2\lambda}$.
   \item[ii)] If $(q_n)_{n\geq1}$ is a sequence of odd positive integers, then 
             $q=-e^{-2\lambda}$.
\end{itemize}
Then $(H_{1},\dots, H_p)$ converges in distribution to a tuple of $q$-Gaussian 
variables $(s_q(h_1),\dots,s_q(h_p))$ for an orthonormal system $h_1,\dots,h_p$. 
Concretely, this means that for every positive integer $k$ and for every 
$\varepsilon:[k]\rightarrow [p]$, we have that
 \begin{equation}\label{limit}
   \lim_{n\rightarrow\infty}
      \E\left[
         \tr(H_{\varepsilon(1)}\cdots H_{\varepsilon(k)})\right] 
            =\sum_{\stackrel{\pi\in \pP_2(k)}{\pi\leq\Ker \varepsilon}} 
                q^{cr(\pi)}
                    =\tau\left(s_q(h_{\ee(1)})\cdots s_q(h_{\ee(k)})\right).
 \end{equation}
Note that all three expression in (\ref{limit}) are zero when $k$ is odd.
\end{theorem}

\begin{corollary}\label{cor:process}
Consider the following dynamical version of the SYK model:
\begin{equation}\label{SYK-model-process}
  H(t):=
    \frac{i^{\lfloor\sfrac{q_n}{2}\rfloor}}{\binom{n}{q_n}^{\sfrac{1}{2}}}
       \sum_{1\leq i_1<\cdots<i_{q_n}\leq n}
           J_{i_1,\dots,i_{q_n}} (t)
              \psi_{i_1}\cdots\psi_{i_{q_n}},
\end{equation}
where the $J_{i_1,\dots,i_{q_n}} (t)$ (with $n\in\mathbb{N}$, $1\leq i_1<\cdots<i_{q_n}\leq n$) are independent classical Brownian motions, and the $q_n$ and $q$ are as in Theorem \ref{MainTheorem}. Then, the process $(H(t))_{t\geq 0}$ converges, for $n\to\infty$, to the $q$-Brownian motion $(S_q(t))_{t\geq 0}$ as given in \eqref{q-BM-eq}, in the sense that we have for all $0\leq t_1,\dots,t_k$ that
\begin{equation}\label{limit-process}
\lim_{n\rightarrow\infty}
      \E\left[
         \tr(H(t_1)\cdots H(t_k))\right] 
                    =\tau\left(S_q(t_1)\cdots S_q(t_k)\right).
\end{equation}
\end{corollary}

Note that Corollary \ref{cor:process} follows from Theorem \ref{MainTheorem} by writing all the appearing $H(t_1),\dots, H(t_k)$ as well as the corresponding $S_q(t_1),\dots, S_q(t_k)$ as sums for orthogonal increments and then expand everything in a multilinear way.

Thus it suffices to prove Theorem \ref{MainTheorem}. For the proof of this we will rely on the following two lemmas. The proof of those will be postponed to 
Section \ref{ProofLemmas}.

\begin{lemma}\label{Lemma1}
For every $\alpha:[k]\rightarrow I_n$ we have the following
\begin{itemize}
 \item[i)] If every block in $\Ker\alpha$ has even size, then 
            $\p_{\alpha(1)}\cdots\p_{\alpha(k)}=\pm I$, where I is the identity matrix.
 \item[ii)] For $\pi\in\pP_2(k)$ with $\Ker\alpha\geq\pi$ we have the identity
             \begin{equation}\label{eq:comm-pi}
                \tr(\p_{\alpha})
                     =(-1)^{q_n cr(\pi)
                         +\sum|\alpha(V)\cap\alpha(W)|},
             \end{equation}
             where the sum is taken over all pairs $\{V,W\}$ of crossing blocks in $\pi$.
             We are using here notation (\ref{N1}) and (\ref{N3}). Also, for $Q,R\in I_n$
             we denote by $Q\cap R$, the set of indices that $Q$ and $R$ have in common.
\end{itemize}
\end{lemma}

Note that, contrary to first impression, the term \eqref{eq:comm-pi} does not depend on $\pi$, even in the case where $\ker\alpha$ is not a pairing itself; for this note that, for example, the contribution for a crossing of two blocks with the same value of $\alpha(V)=\alpha(W)$ is given by
$$(-1)^{q_n\cdot 1+q_n}=(-1)^{2q_n}=1.$$

\begin{lemma}\label{Lemma2}
Let $\pi\in\pP_2(k)$ be a pairing and consider the following sum
\begin{equation}\label{eq:sum-hyper}
S(\pi,n):=\frac{1}{|I_n|^{\sfrac{k}{2}}}
         \sum_{\stackrel{\alpha:[k]\rightarrow I_n}{\Ker\alpha\geq\pi}}
            \tr(\p_\alpha)
\end{equation}
\begin{itemize}
\item[i)]
Then we have the following limit
 \begin{equation}\label{eq:sum-hyper-lim}
\lim_{n\to\infty} S(\pi,n)=q^{cr(\pi)}
 \end{equation}
\item[ii)]
Let $V\in\pi$ be a block of $\pi$. If we fix in \eqref{eq:sum-hyper} the value of $\alpha(V)$ corresponding to this block $V$ and sum only over the remaining $k-2$ $\alpha(i)$'s, then the result of this restricted sum 
\begin{equation}\label{eq:sum-hyper-restr}
\frac{1}{|I_n|^{\sfrac{(k-2)}{2}}}
         \sum_{\substack{\alpha:[k]\rightarrow I_n \\ \Ker\alpha\geq\pi \\
\alpha(V)=R}}
            \tr(\p_\alpha)
 \end{equation}
is the same as in \eqref{eq:sum-hyper}, independent of the chosen block $V$ and of the fixed value $R\in I_n$ for $\alpha(V)$.
\end{itemize}
\end{lemma}
\noindent

\begin{proof}[Proof of Theorem \ref{MainTheorem}]
Consider the following expansion for the left side of (\ref{limit})
 \begin{equation}\label{ExpansionE}
  \E\left[\tr\left(H_\varepsilon\right)\right]
      =\frac{1}{|I_n|^{\sfrac{k}{2}}}
         \sum_{\alpha:[k]\rightarrow I_n}
             \E\left[J^\varepsilon_\alpha\right]
                 \tr\left(\p_\alpha\right).
 \end{equation} 
The variables $\p_R$, for $R\in I_n$, were introduced in (\ref{New_variables}). We are
also using the notation for products of non-commutative variables as introduced in 
(\ref{N1}) and (\ref{N2}).
Thus $H_\ee=H_{\ee(1)}\cdots H_{\ee(k)}$, $J_\alpha^\ee=J_{\alpha(1)}^{\ee(1)}\cdots
J_{\alpha(k)}^{\ee(k)}$, and
$\p_\alpha=\p_{\alpha(1)}\cdots \p_{\alpha(k)}$; note that each $\p_{\alpha(i)}$ is actually a $\p_R$ for $R\in I_n$, i.e., a product of $q_n$-many $\psi$'s.

We can split the sum in (\ref{ExpansionE}) as
\begin{equation*}
\sum_{\alpha:[k]\rightarrow I_n}
   =\sum_{\substack{\alpha:[k]\rightarrow I_n \\ |\Ker\alpha|<k/2}}
        +\sum_{\substack{\alpha:[k]\rightarrow I_n \\ |\Ker\alpha|=k/2}}
            +\sum_{\substack{\alpha:[k]\rightarrow I_n \\ |\Ker\alpha|>k/2}}.
\end{equation*}

The last term does not contribute, because if
$|\Ker\alpha|>k/2$ then $\Ker\alpha$ has a block of size one, and then
$\E\left[J^\varepsilon_\alpha\right]=0$.

By noticing that always $\vert \tr\left(\p_\alpha\right)\vert \leq 1$,
we get for the case $|\Ker\alpha|<k/2$ the bound
\begin{equation}\label{O_bound_E}
 \Bigl\vert \frac{1}{|I_n|^{\sfrac{k}{2}}} 
     \sum_{\substack{\alpha:[k]\rightarrow I_n \\ |\Ker\alpha|<k/2}}
         \E\left[J^\varepsilon_\alpha\right]
             \tr\left(\p_\alpha\right)\Bigr\vert 
                  \leq c_{k} d_k
                      \frac{|I_n|^{\sfrac{k}{2}-1}}{|I_n|^{\sfrac{k}{2}}}
                          =\frac{c_{k}d_k}{|I_n|}.
\end{equation}
The constant $c_k$ comes from the uniform bound condition on the random coefficients in
(\ref{SYK-model}) and $d_k$ is the number of possible $\pi:=\ker\alpha\in\pP(k)$ with $\vert \pi\vert= k/2-1$. The estimate \eqref{O_bound_E} shows that this term does also vanish in the limit $n\to\infty$.

So we are left with the sum over $|\Ker\alpha|=k/2$. 
Since $|\Ker\alpha|=k/2$ we can assume $\Ker\alpha\in\pP_2(k)$, otherwise 
$\Ker\alpha$ has a block of size one, then $\E\left[J^\varepsilon_\alpha\right]=0$.
Also the condition
$\Ker\alpha\in\pP_2(k)$ implies
\[\E\left[J^\varepsilon_\alpha\right]=\left\{\begin{array}{ccc}
                                     1 & \mbox{ if } & \Ker\alpha\leq\Ker\varepsilon,\\
                                     0 & &\mbox{otherwise.} 
                                    \end{array}\right.
\]
This together with Lemma \ref{Lemma2} yields
\begin{align}\label{eq:new-20}
\lim_{n\to\infty} \frac{1}{|I_n|^{\sfrac{k}{2}}}
   \sum_{\pi\in\pP_2(k)}
     \sum_{\stackrel{\alpha:[k]\rightarrow I_n}{\Ker\alpha=\pi}}
       \E\left[J^\varepsilon_\alpha\right]\tr\left(\p_\alpha\right)& \nonumber\\
 =\sum_{\stackrel{\pi\in\pP_2(k)}{\pi\leq\ker\epsilon}}
     q^{cr(\pi)}&
        -\lim_{n\to\infty}\frac{1}{|I_n|^{\sfrac{k}{2}}}
              \sum_{\stackrel{\pi\in\pP_2(k)}{\pi\leq\ker\epsilon}}
                   \sum_{\stackrel{\alpha:[k]\rightarrow I_n}{\Ker\alpha>\pi}}
                         \tr\left(\p_\alpha\right).
\end{align}
For each fixed $\pi\in\pP_2(k)$ with $\pi\leq\ker\epsilon$ we can bound, similar
as in \eqref{O_bound_E}, the corresponding correction term:
\begin{equation}\nonumber
 \Bigl\vert\frac{1}{|I_n|^{\sfrac{k}{2}}}
        \sum_{\stackrel{\alpha:[k]\rightarrow I_n}{\Ker\alpha>\pi}}
             \tr\left(\p_\alpha\right)\Bigr\vert
                 \leq \binom{k/2}2
                       \frac{|I_n|^{k/2-1}}{|I_n|^{k/2}} \to 0\qquad \text{for $n \to\infty$}.
\end{equation}
The binomial factor comes here from the number of possibilities of joining two blocks of $\pi$.

Hence the right hand side of \eqref{eq:new-20} reduces to the first term, which is, by \eqref{eq:Wick-ortho}, the corresponding moment for a $q$-Gaussian family. 
This concludes the proof of Theorem \ref{MainTheorem}.
\end{proof}

\section{Fluctuations and higher order correlations for the multivariate model}\label{section:4}
The classical cumulants are a family $(c_m)_{m\in\mathbb{N}}$ of multilinear  functionals, given by
\begin{equation}\label{Cumulant}
 c_m(a_1,\dots,a_m)
     =\sum_{\sigma\in\pP(m)}
          \E_\sigma\left[a_1,\dots,a_m\right]
              \mu(\sigma,1_m),
\end{equation}
where $\E_\sigma$ stands for
$$\E_\sigma\left[a_1,\dots,a_m\right]
     =\prod_{\substack{B\in\sigma\\B=\{i_1,\dots,i_r\}}}
         \E\left[a_{i_1}\cdots a_{i_r}\right],$$
and $\mu(\sigma,1_m)=(-1)^{|\sigma|-1}(|\sigma|-1)!$ is the M\"obius function.
This family of functionals characterizes tensor independence. See for example 
\cite{book:833269} for more details on this.

In this section we will identify the convergence of 
$c_m(H_{\varepsilon_1},\dots,H_{\varepsilon_m})$,
in a similar way as in Theorem \ref{MainTheorem}. 
Theorem \ref{FluctuationsThm} is an extension of a result
that originally appeared in \cite{Feng:2018tsw}. We will restrict here to the multivariate version for independent copies. The extension of this to the dynamical version follows as before easily via multilinear extension; in order to keep the notation as simple as possible we refrain from giving this dynamical version explicitly.
\begin{theorem}\label{FluctuationsThm}
Let $(H_k)_{k\in\NN}$ be independent copies of the SYK model $H_{n,q_n}$ from Equation \ref{SYK-model},
with centered Gaussian random 
coefficients.
For positive integers $m,k_1,\dots,k_m$ set $k:=k_1+\cdots+k_m$ and denote
$\theta=\{T_1,\dots,T_m\}$, where
$T_1=[1,k_1], T_2=[1+k_1,k_1+k_2],\dots, T_m=[1+k_1+\cdots+k_{m-1},k]$.
Given a function $\varepsilon:[1,k]\rightarrow \NN$, let us denote for each 
$1\leq i\leq m$ the functions $\varepsilon_i:=\varepsilon|_{T_i}$.
Under the same assumptions on $n$ and $q_n$ as in Theorem \ref{MainTheorem}, we have
\begin{equation}\label{Cm-SYK}
\vert I_n\vert^{m-1} \cdot{c_m\left(\tr(H_{\varepsilon_1}),\dots,\tr(H_{\varepsilon_m})\right)}
        \xrightarrow{n\rightarrow\infty}
\sum_{\substack{\pi,\pi'\in\pP_2(k)\\ \pi\vee \theta =1_k, \pi'\leq\theta\\ \pi\leq\ker\ee\\
\vert \pi\vee\pi'\vert=k/2-m+1}}
  q^{cr(\pi')},
\end{equation}
where the parameter $q$ is determined in the same way as in Theorem \ref{MainTheorem}.
\end{theorem}
Note in the above the distinction between $\varepsilon_1$ and $\varepsilon(1)$. The former is the restriction of the function $\varepsilon$ to the set $T_1$, whereas the latter is just the value of $\varepsilon$ at the point 1. Accordingly, on the left hand side of the equation, $H_{\varepsilon_1}$ denotes actually the product $H_{\varepsilon(1)}\cdots H_{\varepsilon(k_1)}$.

\begin{proof}
 By the multilinear property of the cumulant we have
 \begin{align}
  {\vert I_n\vert}^{m-1} 
     c_m\left(\tr(H_{\varepsilon_1}),\dots,\tr(H_{\varepsilon_m})\right)&= \nonumber\\
  \frac{1}{|I_n|^{\frac{k}{2}-m+1}}
     \sum_{\stackrel{\alpha_i:T_i\rightarrow I_n}{1\leq i\leq m}} 
         c_m(J^{\varepsilon_1}_{\alpha_1}&,\dots,J^{\varepsilon_m}_{\alpha_m})
             \tr(\p_{\alpha_1})\cdots\tr(\p_{\alpha_m}). 
 \end{align}

Using the formula for cumulants with products as entries and the Gaussianity
of the random variables $J$ we get

\begin{equation*}
c_m(J^{\varepsilon_1}_{\alpha_1},\dots,J^{\varepsilon_m}_{\alpha_m})=
\sum_{\substack{\pi\in\pP_2(k)\\ \pi\vee \theta =1_k}}
c_\pi(J_{\alpha(1)}^{(\ee(1))},\dots,J_{\alpha(k)}^{(\ee(k))}),
\end{equation*}
and the contribution $c_\pi(\cdots)$ is 1 if $\pi\leq\ker \alpha$ and $\pi\leq \ker\ee$, otherwise it is equal to zero. So we have
\begin{multline}\label{StartFluctuations}
  {\vert I_n\vert}^{m-1} \cdot
     c_m\left(\tr(H_{\varepsilon_1}),\dots,\tr(H_{\varepsilon_m})\right)\\=
  \frac{1}{|I_n|^{\frac{k}{2}-m+1}}
\sum_{\substack{\pi\in\pP_2(k)\\ \pi\vee \theta =1_k\\ \pi\leq\ker\ee}}
     \sum_{\substack{\alpha_i:T_i\rightarrow I_n \\ 1\leq i\leq m \\ \ker\alpha\geq\pi}} 
             \tr(\p_{\alpha_1})\cdots\tr(\p_{\alpha_m}).
 \end{multline}

In order to have non-vanishing contributions in this sum we need that the terms
$\tr(\p_{\alpha_i})$ are different from zero. This will give in leading order additional constraints on the relevant $\alpha$. Since $\vert \tr(\p_{\alpha_i})\vert$ can only take on the values $0$ or $1$, we will take the absolute value of the above equation,
\begin{multline}\label{StartFluctuationsvert}
  {\vert I_n\vert}^{m-1} \cdot
   \vert  c_m\left(\tr(H_{\varepsilon_1}),\dots,\tr(H_{\varepsilon_m})\right)\vert\\\leq
  \frac{1}{|I_n|^{\frac{k}{2}-m+1}}
\sum_{\substack{\pi\in\pP_2(k)\\ \pi\vee \theta =1_k\\ \pi\leq\ker\ee}}
     \sum_{\substack{\alpha_i:T_i\rightarrow I_n \\ 1\leq i\leq m \\ \ker\alpha\geq\pi}} 
             \vert\tr(\p_{\alpha_1})\vert\cdots\vert\tr(\p_{\alpha_m})\vert,
 \end{multline}
and we are trying to identify the leading order contributions to this sum. Note that according to our normalization we expect this to be of order
${|I_n|^{\frac{k}{2}-m+1}}$.

 $\pi$ is a pairing on $T=[1,k]$, but its restriction $\pi_i$ to the interval $T_i$ consists of pairs and singletons; some pairs of $\pi$ connect different intervals and thus those pairs decompose into two singletons under the restriction. In our sum over all $\alpha=(\alpha(1),\dots,\alpha(k))$ we can sum over all $\alpha(j)$ which correspond to pairs of the $\pi_i$; the corresponding factors $\p_R$ in $\tr(\p_{\alpha_i})$ appear then also in pairs and we will, by Lemma \ref{Lemma1}, not change the value of $\vert\tr(\p_{\alpha_i})\vert$ if we cancel the factors; on the other hand for each pair in $\pi$ we have $\vert I_n\vert$ many $R$ to sum over. Hence removing the pairs of all $\pi_i$ will change the problem to an equivalent one and it suffices to consider the situation where all restrictions $\pi_i$ consist just of singletons. Again, in this situation the most
canonical way to achieve that $\tr(\p_{\alpha_i})$ is different from zero is if every factor $\p_R$ of $\tr(\p_{\alpha_i})$ comes in pairs. However, there are now also contributions which are not of this form; however, we expect that they appear in smaller orders - that's what we want to show in the following.

We have now the situation that all $\pi_i$ consist only of singletons. There will be some $i$ for which $T_i$ consists only of one element - however, then $\tr(\p_{\alpha_i})$=0, since $\p_{\alpha_i}$ is then just one of the $\p_R$ and its trace vanishes. So such a situation does not contribute. Next, there will be some $i$ for which $T_i$ consists of exactly two elements. Then $\p_{\alpha_i} $ is equal to a product $\p_R\p_Q$, and the trace of this can only be different from zero if $R=Q$, i.e., we can restrict to summing over the situation that $\ker\alpha_i$ is a pair. As before we can remove this pair  (and identify the two corresponding half-edges) and reduce the problem thus further. So, iterating all these reductions, in the end we are left with a situation where each $T_i$ contains at least three elements. Now we have to estimate more seriously. Let us check what happens when we remove again one the intervals together will all its $\alpha's$. Note that the other ends of the removed singletons ly in some other intervals, thus the corresponding $\alpha$ there can not be summed over in the following, but its value will be fixed. Hence, on such an interval $T_j$ we will have to control in the following an expression of the form $\tr(\p_{\alpha_j}\p)$,
where $\p$ can be any arbitrary product of $\psi$'s and we sum over $\alpha_j$. Let us now control the effect of removing such a $T_j$. If $T_j$ contains $r$ elements, we are removing the sum over $2r$ many $\alpha$'s and we are reducing the number of intervals by 1. What is the contribution of the removed $\alpha$'s in the original sum? Since only terms with $\tr(\p_{\alpha_j}\p)\not=$ are relevant, we cannot sum over all the $r$ $\alpha$'s independently, but if we choose the values of the first $r-1$ ones, then the remaining $\alpha$ is fixed (where it could be that this fixed value is not a valid product of $q_n$-many different $\psi$'s), thus we have removed at most a contribution of $\vert I_n\vert^{r-1}$. So this reduction is again consistent with our highest order conjecture. We just continue with all our reductions since nothing is left. Of course, this is not good as we have then just estimated everything just by the hightest order. If we want to get rid of some terms we have to improve our estimate somewhere. Actually we do this in our first estimate for the case where all $T_i$ have at least three elements. Above we said that in the reduction step we removed the contribution of at most
$\vert I_n\vert^{r-1}$-many $\alpha$'s, because we have to make sure that
$\tr(\p_{\alpha_j}\p)$ is different from zero. In the first step we have, however, no extra $\p$ there (as this is only arising in the removals of three or more elements), thus we have to estimate the numer of $(R_1,\dots,R_m)\in I_n^m$ for which $\tr(\p_{R_1}\cdots \p_{R_m})\not=0$. This was done in \cite{Feng:2018tsw}, giving in this case as an improved (non-trivial) estimate the factor $c\vert I_n\vert^{r-1} n^{-1/2}$, for some absolute constant $c$. (We expect that the same might be true if we consider $\tr(\p_{\alpha_j}\p)$ instead of $\tr(\p_{\alpha_j})$, but this is not needed for our arguments.) Thus, if we arrive in our reduction at a situation where we have three or more singletons in an interval, we get an extra factor $n^{-1/2}$, which will let this distribution disappear asymptotically. But this means that we can restrict in our summation over the $\alpha$'s to a situation where we also have pairings in all $\ker\alpha_i$ - hence we have pairings $\pi'_i\in \pP_2(T_i)$ such that
$\pi'_i\leq \ker \alpha_i$ for all $1\leq i\leq m$. If we denote by $\pi'$ the union over the blocks of all $\pi'_i$ for all $i$, then we have of course
$\pi'\leq\ker\alpha$, hence our condition for the summation over $\alpha$ is now
$\pi\vee\pi'\leq \ker\alpha$. This gives for the summation over $\alpha$ the order $\vert I_n\vert^{\vert \pi\vee\pi'\vert}$.  
Note that we have always the inequality $\vert \pi\vee\pi'\vert\leq k/2-m+1$,
because $\pi'$ with its $k/2$ blocks lives on $m$ different intervals $T_i$, and those different intervals have, by the condition $\pi\vee\theta=1_k$, to be connected by $\pi$. For this we need at least $m-1$ pairs of $\pi$ to make those connections. This shows on one hand that the maximal order in the summation in \eqref{StartFluctuationsvert} can not be bigger than
$\vert I_n\vert^{k/2-m+1}$, and on the other hand it implies that we need for $\pi'$ the additional constraint that
$\vert \pi\vee\pi'\vert=k/2-m+1$ to achieve this leading order. Thus, going back to \eqref{StartFluctuations}, we have seen that asymptotically in $n$ the leading order of 
  ${\vert I_n\vert}^{m-1} \cdot
     c_m\left(\tr(H_{\varepsilon_1}),\dots,\tr(H_{\varepsilon_m})\right)$
is given by
\begin{equation}\label{eq:leading-order}
  \frac{1}{|I_n|^{\frac{k}{2}-m+1}}
\sum_{\substack{\pi\in\pP_2(k)\\ \pi\vee \theta =1_k\\ \pi\leq\ker\ee}}
\sum_{\substack{\pi'\in\pP(k)\\ \pi'\leq\theta\\
\vert \pi\vee\pi'\vert=k/2-m+1}}
     \sum_{\substack{\alpha:[k]\rightarrow I_n \\ \ker\alpha\geq\pi\vee\pi'}} 
             \tr(\p_{\alpha_1})\cdots\tr(\p_{\alpha_m}). 
\end{equation}

We have now to calculate the asymptotics of \eqref{eq:leading-order}. One has to note that the restrictions $\alpha_i$ to the intervals $T_i$ are not independent from each other because some of their values are coupled by the condition $\ker\alpha\geq \pi$. Without this condition and without the power $m-1$ at $\vert I_n\vert$, this sum would just decouple into a product of sums for each of the
$\tr(\p_{\alpha_i})$. Our goal is to see that this decoupling effectively still happens in leading order. For this it is important that, by the condition $\vert \pi\vee\pi'\vert=k/2-m+1$, we have as few as possible connections, under the connectedness condition, between the different $T_i$. In order to understand what this implies, let us have a closer look on the $\pi$ appearing in \eqref{eq:leading-order} as follows. The restrictions $\pi_i$ to the $T_i$ consist of pairs and singletons. Let us denote the total number of singletons by $k_s$ and the total number of pairs by $k_p$; then we have $k=k_s+2 k_p$. We claim now that we must find at least one $\pi_i$ which has only two singletons. Namely, assume to the contrary that each $\pi_i$ has at least four singletons, then we have that $k_s\geq 4m$. But then 
$$\vert \pi\vee\pi'\vert=k/2-m+1=k_p+k_s/2-m+1
\geq k_p +k_s/2-k_s/4+1=k_p+k_s/4+1.$$ 
On the other hand, each block of $\pi\vee\pi'$ contains at least two elements, and if it contains a singleton then it contains at least four elements; hence there can be at most $k_p+k_s/4$ blocks of $\pi\vee\pi'$. This contradiction shows that we cannot start with all intervals having at least four singletons. Thus there must be at least one $\pi_i$ which has only two singletons. Similar as above, one also sees that those two singletons of $\pi_i$ must be connected by a block of $\pi'_i$ and that for the other elements in $T_i$ the blocks of $\pi_i$ and the blocks of $\pi'_i$ have to agree, in order to achieve the maximal value of 
$\vert \pi\vee\pi'\vert$.
But then in our summation over $\alpha_i$ we are in the situation of the restricted summation, Equation \eqref{eq:sum-hyper-restr}, of Lemma \ref{Lemma2}: we have one block $V$ of $\pi'$ for which $\alpha(V)$ is determined by values from other intervals, but otherwise we sum just over $\ker \alpha_i\geq\pi'$. Since this summation is, by Lemma \ref{Lemma2}, independent of the fixed value of $\alpha(V)$, we can decouple this value from the outside intervals and just sum over all possibilities for $\alpha(V)$. This summation will then give an extra factor $\vert I_n\vert$. We can iterate this now, by finding another $T_j$ of the remaining intervals, for which $\pi_j$ (after having removed $T_i$) has only two singletons and decouple the summation for this $T_j$ as before from the rest.
In each of the $m-1$ decoupling steps we get an extra factor $\vert I_n\vert$ and in the end we have removed all the constraints $\ker\alpha\geq \pi$ and are left with just doing summations over the $T_i$'s separately according to the constraint $\ker\alpha\geq \pi'$; i.e., \eqref{eq:leading-order} is the same as

\begin{align*}
  &\frac{1}{|I_n|^{\frac{k}{2}}}
\sum_{\substack{\pi\in\pP_2(k)\\ \pi\vee \theta =1_k\\ \pi\leq\ker\ee}}
\sum_{\substack{\pi'\in\pP(k)\\ \pi'\leq\theta\\
\vert \pi\vee\pi'\vert=k/2-m+1}}
     \sum_{\substack{\alpha:[k]\rightarrow I_n \\ \ker\alpha\geq\pi'}} 
             \tr(\p_{\alpha_1})\cdots\tr(\p_{\alpha_m})\\[0.4em]
&=
\sum_{\substack{\pi,\pi'\in\pP_2(k)\\ \pi\vee \theta =1_k, \pi'\leq\theta\\ \pi\leq\ker\ee\\
\vert \pi\vee\pi'\vert=k/2-m+1}}
     \Bigl(\frac{1}{|I_n|^{\frac{\vert T_1\vert}{2}}}
\sum_{\substack{\alpha_1:[\vert T_1\vert]\rightarrow I_n \\ \ker\alpha_1\geq\pi_1'}} 
             \tr(\p_{\alpha_1})\Bigr)\cdots
\Bigl(\frac{1}{|I_n|^{\frac{\vert T_m\vert}{2}}}
\sum_{\substack{\alpha_1:[\vert T_m\vert]\rightarrow I_n \\ \ker\alpha_m\geq\pi_m'}} 
\tr(\p_{\alpha_m})\Bigr)\\[0.4em]
&\to\sum_{\substack{\pi,\pi'\in\pP_2(k)\\ \pi\vee \theta =1_k, \pi'\leq\theta\\ \pi\leq\ker\ee\\
\vert \pi\vee\pi'\vert=k/2-m+1}}
  q^{cr(\pi_1')}\cdots q^{cr(\pi_m')}
=\sum_{\substack{\pi,\pi'\in\pP_2(k)\\ \pi\vee \theta =1_k, \pi'\leq\theta\\ \pi\leq\ker\ee\\
\vert \pi\vee\pi'\vert=k/2-m+1}}
  q^{cr(\pi')}\qquad\text{for $n\to\infty$}
\end{align*}
\end{proof} 

\begin{remark}
\begin{enumerate}
\item
In the proof of Theorem \ref{FluctuationsThm} we have seen the nature of the pairs $(\pi,\pi')$ of pair-partitions which contribute to the fluctuations. Each block of $\pi'$ lives on one of the $T_i$, whereas $\pi$ has also some blocks which connect different intervals. In order to get the maximal number of blocks of $\pi\vee\pi'$ we need that the blocks of $\pi$ which do not connect different intervals are also blocks of $\pi'$, The following picture shows a typical contribution for a $6^{th}$ order cumulant. It is natural to draw the intervals $T_i$ as circles, since each of them corresponds to a trace. The blocks of $\pi'$ are drawn as dashed lines within the circles, the blocks of $\pi$ as solid lines outside of the the circles. 
One sees that the effect of those blocks of $\pi$ which connect different circles is to collect several blocks of $\pi'$ together; thus one can also identify the pair $(\pi,\pi')$ with a partition of the blocks of $\pi'$. In this form our formula for the fluctuations of the sparse SYK random matrix looks quite similar to some of the  terms showing up in the description of the fluctuations of a GUE random matrix in terms of partitioned permutations \cite{CMSS}. This suggests that there might be a more general theory of higher order freeness which addresses both the usual GUE as well as sparse random matrices. We will follow up on this connection in future investigations.

\vskip-1cm\hskip-3cm
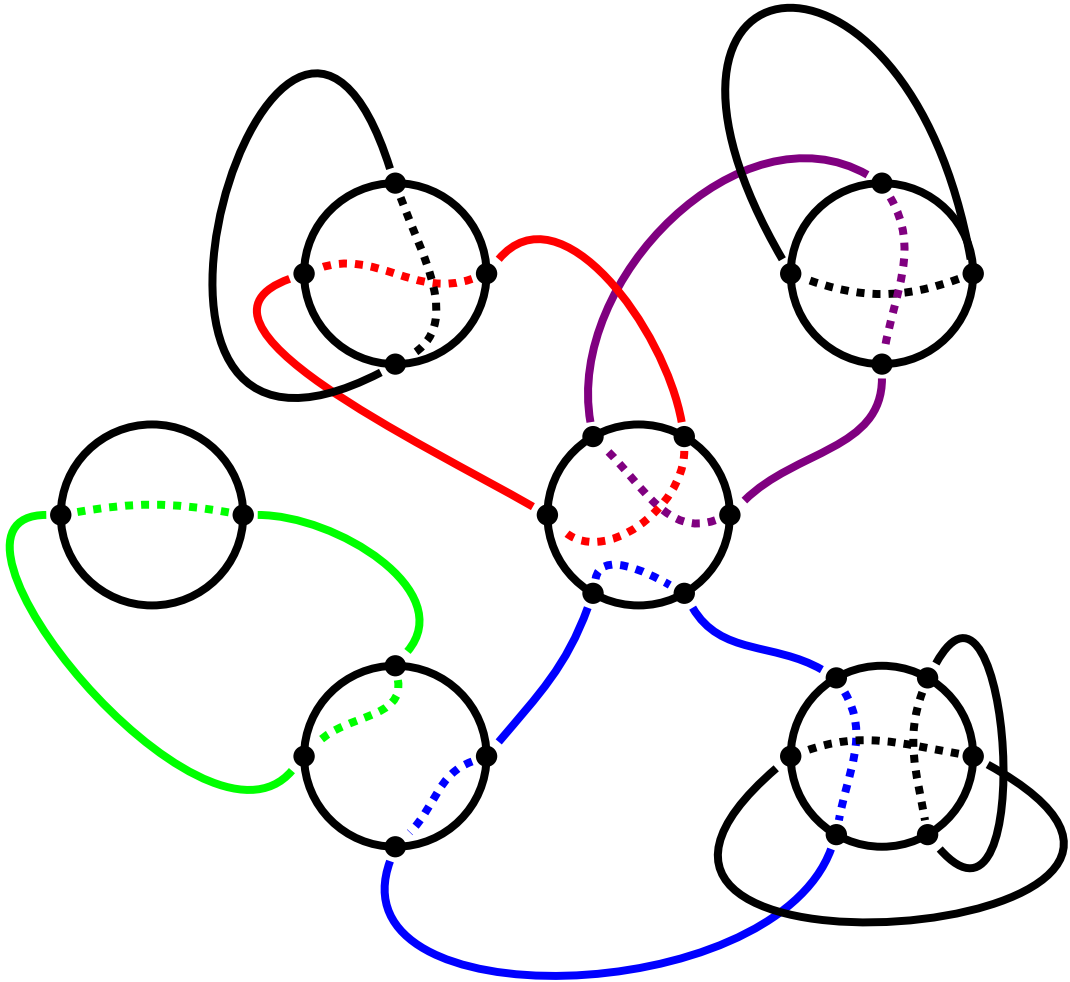
\begin{tikzpicture}[baseline={(-3,-0.5)}]

   \begin{scope}[shift={(-5,-5cm)},scale=0.4, line width=3pt]
\draw (0,0) circle (3cm);
\node (n1) at (0*360:3cm) {};
\node (n2) at (360/6:3cm) {};
\node (n3) at (2*360/6:3cm) {};
\node (n4) at (3*360/6:3cm) {};
\node (n5) at (4*360/6:3cm) {};
\node (n6) at (5*360/6:3cm) {};
\fill[black] (n1) circle[radius=10pt];
\fill[black] (n2) circle[radius=10pt];
\fill[black] (n3) circle[radius=10pt];
\fill[black] (n4) circle[radius=10pt];
\fill[black] (n5) circle[radius=10pt];
\fill[black] (n6) circle[radius=10pt];

\draw (8,8) circle (3cm);
\node (m1) at  ([xshift=8cm,yshift=8cm] 0*360/4:3cm) {};
\node (m2) at  ([xshift=8cm,yshift=8cm] 1*360/4:3cm) {};
\node (m3) at  ([xshift=8cm,yshift=8cm] 2*360/4:3cm) {};
\node (m4) at  ([xshift=8cm,yshift=8cm] 3*360/4:3cm) {};
\fill[black] (m1) circle[radius=10pt];
\fill[black] (m2) circle[radius=10pt];
\fill[black] (m3) circle[radius=10pt];
\fill[black] (m4) circle[radius=10pt];

\draw (-8,8) circle (3cm);
\node (a1) at  ([xshift=-8cm,yshift=8cm] 0*360/4:3cm) {};
\node (a2) at  ([xshift=-8cm,yshift=8cm] 1*360/4:3cm) {};
\node (a3) at  ([xshift=-8cm,yshift=8cm] 2*360/4:3cm) {};
\node (a4) at  ([xshift=-8cm,yshift=8cm] 3*360/4:3cm) {};
\fill[black] (a1) circle[radius=10pt];
\fill[black] (a2) circle[radius=10pt];
\fill[black] (a3) circle[radius=10pt];
\fill[black] (a4) circle[radius=10pt];

\draw (-8,-8) circle (3cm);
\node (c1) at  ([xshift=-8cm,yshift=-8cm] 0*360/4:3cm) {};
\node (c2) at  ([xshift=-8cm,yshift=-8cm] 1*360/4:3cm) {};
\node (c3) at  ([xshift=-8cm,yshift=-8cm] 2*360/4:3cm) {};
\node (c4) at  ([xshift=-8cm,yshift=-8cm] 3*360/4:3cm) {};
\fill[black] (c1) circle[radius=10pt];
\fill[black] (c2) circle[radius=10pt];
\fill[black] (c3) circle[radius=10pt];
\fill[black] (c4) circle[radius=10pt];

\draw (8,-8) circle (3cm);
\node (b1) at  ([xshift=8cm,yshift=-8cm] 0*360/6:3cm) {};
\node (b2) at  ([xshift=8cm,yshift=-8cm] 1*360/6:3cm) {};
\node (b3) at  ([xshift=8cm,yshift=-8cm] 2*360/6:3cm) {};
\node (b4) at  ([xshift=8cm,yshift=-8cm] 3*360/6:3cm) {};
\node (b5) at  ([xshift=8cm,yshift=-8cm] 4*360/6:3cm) {};
\node (b6) at  ([xshift=8cm,yshift=-8cm] 5*360/6:3cm) {};
\fill[black] (b1) circle[radius=10pt];
\fill[black] (b2) circle[radius=10pt];
\fill[black] (b3) circle[radius=10pt];
\fill[black] (b4) circle[radius=10pt];
\fill[black] (b5) circle[radius=10pt];
\fill[black] (b6) circle[radius=10pt];

\draw (-16,0) circle (3cm);
\node (d1) at  ([xshift=-16cm,yshift=0cm] 0*360/4:3cm) {};
\node (d2) at  ([xshift=-16cm,yshift=0cm] 2*360/4:3cm) {};
\fill[black] (d1) circle[radius=10pt];
\fill[black] (d2) circle[radius=10pt];

\draw[violet, dashed] (n1) to[out=200,in=-45] (n3);
\draw[red, dashed] (n2)  to[out=270,in=-45]  (n4);
\draw[dashed, blue] (n5) to[out=80,in=150] (n6);

\draw [violet] (n3) to[out=100,in=150] (m2);
\draw [violet] (n1) to[out=45,in=270] (m4);

\draw [red] (n2) to[out=100,in=50] (a1);
\draw [red] (n4) to[out=150,in=200] (a3);

\draw [blue] (c1) to[out=50,in=250] (n5);
\draw [blue] (c4) to[out=250,in=250] (b5);
\draw [blue] (n6) to[out=300,in=150] (b3);

\draw [dashed, blue] (c1) to[out=200,in=45] (c4);
\draw [dashed,green] (c2) to[out=280,in=45] (c3);
\draw [blue, dashed] (b3) to[out=300,in=80] (b5);
\draw [dashed] (b1)  to[out=170,in=20] (b4);
\draw [dashed] (b2)  to[out=250,in=100] (b6);

\draw[bend right=30] (b2) to[out=150,in=400,distance=5cm] (b6);
\draw[bend right=30] (b1) to[out=150,in=400,distance=12cm] (b4);

\draw [dashed, green] (d1) to[out=170,in=10] (d2);

\draw [green] (c2) to[out=50,in=0] (d1);
\draw [green] (c3) to[out=230,in=180] (d2);

\draw[dashed] (m1) to[out=200,in=-20] (m3);
\draw[violet, dashed] (m2) to[out=300,in=80] (m4);

\draw[bend right=30] (m1) to[out=280,in=300,distance=12cm] (m3);

\draw [dashed, red] (a1) to[out=200,in=20] (a3);
\draw [dashed] (a2) to[out=290,in=30] (a4);
\draw[bend right=80] (a2) to[out=200,in=300,distance=12cm] (a4);

\end{scope}
\end{tikzpicture}

\item
One should notice that in the contribution of a configuration $(\pi,\pi')$ to \eqref{Cm-SYK} only $\pi'$ is involved, via its number of crossings. Those number of crossings is a well-defined quantity since its calculation factorizes into the number of crossings for each of the circles. For $\pi$ we do not have to care about its crossings - which is a good thing as the number of crossings of a multi-annular pairing is not really well defined. On the other hand, $\pi$ governs the constraint $\pi\leq\ker\ee$.
\end{enumerate}
\end{remark}

\section{Proof of the lemmas}\label{ProofLemmas}
\begin{proof}[Proof of Lemma \ref{Lemma0}]
For $R=(i_1,\dots,i_{q_n})\in I_n$, a direct computation yields
\begin{align*}
\p^2_R&=(\psi_{i_{1}}\cdots\psi_{i_{q_n}})(\psi_{i_{1}}\cdots\psi_{i_{q_n}})
        i^{2\lfloor\frac{q_n}{2}\rfloor}\\
      &=(-1)^{\lfloor\frac{q_n}{2}\rfloor}
        (-1)^{q_n-1}(-1)^{q_n-2}\cdots(-1)^{q_n-q_n} I\\
      &=(-1)^{\lfloor\frac{q_n}{2}\rfloor}(-1)^{\frac{q_n(q_n-1)}{2}}\\&=I,
\end{align*}
where $I$ is the identity matrix. In the last equation we used that 
$\lfloor\frac{q_n}{2}\rfloor$ and $\frac{q_n(q_n-1)}{2}$ have the same parity.\\
Now let $R=(i_1,\dots,i_{q_n})$ and $Q=(j_1,\dots,j_{q_n})$ be in $I_n$, observe that
\[
(\psi_{i_{1}}\cdots\psi_{i_{q_n}})\psi_{j_1}=\left\{
\begin{array}{ccc}
 \psi_{j_1}(\psi_{i_1}\cdots\psi_{i_{q_n}})(-1)^{q_n}&\mbox{ if } & j_1\notin\{i_1,\dots,i_{q_n}\}\\
  & & \\
 \psi_{j_1}(\psi_{i_1}\cdots\psi_{i_{q_n}})(-1)^{q_n+1}& \mbox{ if } & j_1\in\{i_1,\dots,i_{q_n}\}
\end{array}\right. 
\]
Then by iteration we get
\begin{align*}
\p_R\p_Q&=i^{2\lfloor\frac{q_n}{2}\rfloor}
         (\psi_{i_1}\cdots\psi_{i_{q_n}})(\psi_{j_1}\cdots\psi_{j_{q_n}})\\
        &=i^{2\lfloor\frac{q_n}{2}\rfloor}
         (-1)^{q_n^2+|Q\cap R|}
         (\psi_{j_1}\cdots\psi_{j_{q_n}})(\psi_{i_1}\cdots\psi_{i_{q_n}})\\
        &=(-1)^{q_n+|Q\cap R|}\p_Q\p_R,
\end{align*}
where $|Q\cap R|$ stands for the number of common indices in $Q$ and $R$. 
\end{proof}

\begin{proof}[Proof of Lemma \ref{Lemma1}]
For $\alpha:[k]\rightarrow I_n$ with $\Ker \alpha=\{V_1,\cdots,V_r\}$, it follows from 
the anti-commutation relation (\ref{psi_commutation}) that 
\begin{align}
 \p_{\alpha(1)}\cdots\p_{\alpha(k)}
   &=\pm \p^{|V_1|}_{\alpha(V_1)}\cdots\p^{|V_r|}_{\alpha(V_r)}. \label{Boldprod1}
\end{align}
The notation $\alpha(V)$ was introduced in (\ref{N3}).
\begin{itemize}
\item[i)] It follows from (\ref{psi_square}) and (\ref{Boldprod1}) that, if
            the $|V_1|,\dots,|V_r|$ are all even, then $\p^{|V_i|}_{\alpha(V_i)}=I$.
\item[ii)] We now assume $\Ker\alpha\geq \pi\in\pP_2(k)$ and we want to determine the sign 
           in (\ref{Boldprod1}). 
           If $\pi\in NC_2(k):=\{\sigma\in\pP_2(k)|cr(\sigma)=0\}$,
           then $\p_\alpha=I$. This comes from the iterative characterization of elements          
           in $NC_2(k)$; see, for example, \cite[Remark: 9.2]{book:833269}.
           If $\pi\notin NC_2(k)$, then we need to apply the relation 
           (\ref{psi_commutation}) to each crossing in $\pi$, and reduce 
           $\p_{\alpha(1)}\cdots\p_{\alpha(k)}$ to the identity. In this processes we 
           obtain $(-1)^{q_n+|\alpha(V)\cap\alpha(W)|}$ for each pair $\{V,W\}$ of
           crossing blocks in $\pi$.

Note that for the above arguments it is not necessary that $\ker\alpha$ is itself a pairing.
\end{itemize} 
\end{proof}

\begin{proof}[Proof of Lemma \ref{Lemma2}]
\begin{itemize}
\item[i)]
This is proved in \cite{Erd-2014,Feng:2018zsx}. Note that the case $q=0$, proved in \cite{Feng:2018zsx}, is more involved than the other cases, because then the possible intersections between $Q$ and $R$ in $I_n$ are typically much larger than in the other cases, and thus harder to control.

\item[ii)] It suffices to show that for a block $V$ the restricted sum in \eqref{eq:sum-hyper-restr} is independent of the choice of $\alpha(V)=R$. For two possible choices $R_1=(i_1,\dots,i_{q_n}), R_2=(j_1,\dots,j_{q_n})\in I_n$ we take a bijection 
$\tau:[n]\to[n]$ which maps $i_k$ to $j_k$ for $1\leq k\leq q_n$ and is arbitrary otherwise. We extend this to a map $\tau:I_n\to I_n$ by declaring $\tau(R)$ for $R=(t_1,\dots,t_{q_n})\in I_n$ as the tuple in $I_n$ which we get by ordering the numbers $\tau(t_1),\dots,\tau(t_{q_n})$. This $\tau$ maps then $R_1$ to $R_2$ and has the property that it preserves the size of intersections, i.e., 
$\vert Q\cap R\vert=\vert\tau(Q)\cap \tau(R)\vert$ for all $Q,R\in I_n$ and thus also in particular $\tr(\p_\alpha)=\tr(\p_{\tau\circ\alpha})$. 
If we apply this bijection to our restricted sum, then the restricted sum for $R_1$ is transformed into the restricted sum for $R_2$:
\begin{align*}
\frac{1}{|I_n|^{\sfrac{(k-2)}{2}}}
         \sum_{\substack{\alpha:[k]\rightarrow I_n \\ \Ker\alpha\geq\pi \\
\alpha(V)=R_2}}
            \tr(\p_\alpha)
&=
\frac{1}{|I_n|^{\sfrac{(k-2)}{2}}}
         \sum_{\substack{\tau\circ\beta:[k]\rightarrow I_n \\ \Ker\tau\circ \beta\geq\pi \\
\tau\circ\beta(V)=R_2}}
            \tr(\p_{\tau\circ\beta})\\
&=
\frac{1}{|I_n|^{\sfrac{(k-2)}{2}}}
         \sum_{\substack{\beta:[k]\rightarrow I_n \\ \Ker\beta\geq\pi \\
\beta(V)=R_1}}
            \tr(\p_{\beta})
 \end{align*}
\end{itemize}
\end{proof}

\section{On the analytic description of the multivariate $q$-Gaussian distribution}

We have established in Theorem \ref{MainTheorem} that one can describe the limit of independent SYK models by concrete operators $s_q(h)$ on the $q$-deformed Fock space. This allows to give operator realizations, via \eqref{limit}, for the limits of expectation values in the SYK model. Unfortunately, this does not imply that we have in the case $p>1$ a good analytic description of the limit object. The relevant analytic object in this context is the operator-valued Cauchy transform, which is defined as follows. Consider $X_i:=s_q(h_i)$ ($i=1,\dots,p$), for some orthonormal $h_1,\dots,h_p$. In order to deal with the distribution of the tuple $(X_1,\dots,X_p)$ we put those $p$ operators as diagonal elements into an $p\times p$ matrix
\begin{equation}
X=\begin{pmatrix}\label{eq:matrix-X}
X_1& 0&\dots&0\\
0&X_2&\dots&0
\\
\vdots&\vdots&\ddots&\vdots\\
0&0&\hdots&X_p
\end{pmatrix},
\end{equation}
put $\cB:=M_p(\CC)$,
and then define the operator-valued Cauchy transform $G_X=(G_X^{(k)})_{k\in\NN}$ of this
as the collection of all functions
\begin{align*}
G_X^{(k)}:H^+(M_k(\cB))&\to H^-(M_k(\cB))\\
z&\mapsto \text{id}\otimes\tau [(z-1\otimes X)^{-1}],
\end{align*}
where $H^\pm$ denote the upper and lower, respectively, halfplane in the considered operator algebras (given by requiring that the imaginary part of the operators 
are strictly positive and negative, respectively). For more information and precise definitions of such non-commutative functions, we refer to \cite{KV,Jek}.
This Cauchy transform is a well defined analytic function which contains all information about the distribution of the tuple $(X_1,\dots,X_p)$ -- in particular, the expectation values as in \eqref{limit} can be recovered as the coefficients in the Taylor expansion of those functions about infinity. The problem is that we do not have any nice concrete analytic description of this function. In the case $p=1$ of just one operator $s_q$ (where we know quite a bit about the limit distribution) one has, for example, a good continued fraction expansion of the Cauchy transform $G$ (which in this case is just an ordinary analytic function from $\CC^+$ to $\CC^-$) in the form
$$G(z)=\cfrac 1{z-\cfrac 1{z- \cfrac {1+q}{ z-\cfrac {1+q+q^2}{z-\dots}}}}.$$
The naive guess that one might also have a corresponding operator-valued version of such a continued fraction expansion is unfortunately not true. Whereas in the scalar case any distribution has a continued fraction expansion for its Cauchy transform, this does not hold any more in the operator-valued setting (see \cite{AW}), and it is easy to check that the matrix $X$ in \eqref{eq:matrix-X} for the $q$-Gaussian distribution is one of the basic examples where this fails, 

This absence of a nice analytic description of the multivariate $q$-Gaussian distribution is the main reason that our progress on a deeper understanding of this distribution (e.g., for addressing free entropy or Brown measure questions in this context) is quite slow. Also the calculations of the 2- and 4-point functions of the SYK model in \cite{Berk1,Berk2} might benefit from such a better analytic understanding. 
It remains to be seen whether the link between our dynamical version of the SYK model and the $q$-Brownian motion leads to progress on such questions.

\section*{Acknowledgements}
R.S. thanks Li Han, Renjie Feng, and Micha Berkooz for discussions about the SYK model.
We also thank the anonymous reviewers whose comments helped to improve  the manuscript substantially.

This work has been supported by the ERC Advanced Grant NCDFP 339760 and by the SFB-TRR 195, Project I.11.

\bibliographystyle{srt}

\begin{thebibliography}{1}

\bibitem{AW}
M. Anshelevich and J. Williams:
\newblock{ Operator-valued Jacobi parameters and examples of operator-valued distributions.}
\newblock{\em Bulletin des Sciences Math\'ematiques,} Vol 145 (2018), 1--37.

\bibitem{Berk1}
M. Berkooz, M. Isachenkov, V. Naravlansky, and T. Torrents: 
\newblock{Towards a full solution of the large N double-scaled SYK model.}
\newblock{\em Journal of High Energy Physics} (2019) 2019: 79.

\bibitem{Berk2}
M. Berkooz, P. Narayan, and J. Simon: 
\newblock{ Chord diagrams, exact correlators in spin glasses and black hole bulk reconstruction.}
\newblock{\em Journal of High Energy Physics}  (2018) 2018: 192.

\bibitem{Mare-1991}
M. Bozejko, R. Speicher:
\newblock{ An example of a generalized brownian motion.}
\newblock{\em Communications in Mathematical Physics,} 137, (1991), 519--531.

\bibitem{BKS}
M. Bozejko, B. K\"ummerer, and R. Speicher:
\newblock{ q-Gaussian Processes: Non-commutative and Classical Aspects.}
\newblock {\em Communications in Mathematical Physics} 185(1), (1997), 129--154.

\bibitem{CMSS}
B. Collins, J. Mingo, P. Sniady, and R. Speicher:
\newblock{Second order freeness and fluctuations of random matrices. III. Higher order freeness and free cumulants.}
\newblock{\em Documenta Mathematica,} 12, (2017), 1--70.

\bibitem{Erd-2014}
L. Erd\"os and D. Schr\"oder:
\newblock Phase transition in the density of states of quantum spin glasses.
\newblock {\em Mathematical Physics, Analysis and Geometry}, 17:9164, (2014).

\bibitem{EY}
L. Erd\"os and HT Yau:
Dynamical approach to random matrix theory.
Vol. 28, {\em Courant Lecture Notes in Mathematics}, 2017.

\bibitem{Feng:2018zsx}
R. Feng, G. Tian, and D. Wei:
\newblock {Spectrum of SYK model}.
\newblock{\em Peking Mathematical Journal} (2019) 2: 41.

\bibitem{Feng:2018tsw}
R. Feng, G. Tian, and D. Wei:
\newblock {Spectrum of SYK model II: Central limit theorem}.
\newblock arXiv:1806.05714.

\bibitem{Verb}
A. Garcia-Garcia, Y. Jia, J. Verbaarschot: 
\newblock{Exact moments of the Sachdev-Ye-Kitaev model up to order $1/N^2$}.
\newblock{\em Journal of High Energy Physics}  (2018) 2018: 146.

\bibitem{Hall}
B. Hall: PDE methods in random matrix theory.
arXiv:1910.09274.

\bibitem{Jek}
D. Jekel:
Operator-valued non-commutative probability. 
\newblock{Preprint, available at }\url{https://www.math.ucla.edu/~davidjekel/projects.html}, November, 2018.

\bibitem{KV}
D. Kaliuzhniyi-Verbovetskyi and V. Vinnikov: Foundations of non-commutative function theory. 
\newblock{\em Mathematical Surveys and Monographs,} vol. 199 (American Mathematical Society,
Providence, RI, 2014)

\bibitem{Kitaev}
A. Kitaev:
\newblock A simple model of quantum holography.
\newblock{\url{http://online.kitp.ucsb.edu/online/entangled15/kitaev/}}, talk, April 2015.

\bibitem{MalSt}
J. Maldacena and D. Stanford: 
\newblock{Comments on the Sachdev--Ye--Kitaev model}.
\newblock{\em Physical Review D}, 94, (2016) 

\bibitem{MS}
J. Mingo and R. Speicher: 
Free Probability and Random Matrices.
\newblock{\em Fields Institute Monographs 35,} Springer, New York; Fields Institute for Research in Mathematical Sciences, Toronto, ON, 2017.

\bibitem{book:833269}
A. Nica and R. Speicher:
\newblock { Lectures on the Combinatorics of Free Probability}.
\newblock {\em London Mathematical Society Lecture Note Series 335,} Cambridge
  University Press, 2006.

\bibitem{Parisi}
G. Parisi:
\newblock{$D$-dimensional arrays of Josephson junctions, spin glasses and $q$-deformed harmonic oscillators}.
\newblock{\em Journal of Physics A: Mathematical and General,} 27, no. 23 (1994): 7555.


\bibitem{Ross}
S. M. Ross:
\newblock{A first course in probability, 8th ed.}
\newblock{Pearson Prentice Hall, Pearson Education, Inc.} (2010).

\bibitem{Sachdev-1993}
S. Sachdev and J. Ye:
\newblock Gapless spin-fluid ground state in a random quantum Heisenberg magnet.
\newblock {\em Physical Review Letters}, 70, 5 (1993).

\bibitem{Sniady}
P. Sniady: 
\newblock Gaussian random matrix models for $q$-deformed Gaussian random variables. 
\newblock{\em Communications in Mathematical Physics,} 216, (2001), 515--537

\bibitem{SpCLT}
R. Speicher: 
A non-commutative central limit theorem.
\newblock{\em Mathematische Zeitschrift,} 209, (1992), 55–-66.


\end{thebibliography}

\end{document}